\documentclass[12pt]{article}

\usepackage{amsfonts,amssymb,graphicx}
\usepackage{amsmath}
\usepackage{color}
\usepackage[colorlinks=true, linkcolor=blue,citecolor=blue,urlcolor=blue]{hyperref}
\usepackage{hyperref}
\usepackage{amsmath, amscd,url}

\newtheorem{thm}{Theorem}[section]
\newtheorem{lem}[thm]{Lemma}
\newtheorem{proposition}[thm]{Proposition}
\newtheorem{cor}[thm]{Corollary}
\newtheorem{predef}{Definition}
\newenvironment{defn}{\predef\rm}{\endpredef}
\newtheorem{preex}{Example}
\newenvironment{example}{\preex\rm}{\endpreex}
\newtheorem{prerem}{Remark}
\newenvironment{rem}{\prerem\rm}{\endprerem}
\newtheorem{preprob}{Problem}
\newenvironment{prob}{\preprob\rm}{\endpreprob}
\newcommand{\qed}{\qquad$\Box$}
\newtheorem{unnumber}{}

\newcommand{\sd}{\mathbin{\triangle}}

\newtheorem{prepf}[unnumber]{Proof}
\newenvironment{proof}{\prepf\rm}{\endprepf}

\newcommand{\Ann}{\mathop{\mathrm{Ann}}}

\begin{document}

\title{Induced subgraphs of zero-divisor graphs}
\author{G. Arunkumar\\
{\small Department of Mathematics,  Indian Institute of Technology Madras,
Chennai 600032, India}\\[2mm]
Peter J. Cameron\\
{\small School of Mathematics and Statistics, University of St Andrews,
Fife, KY16 9SS, UK}\\
{\small ORCID: 0000-0003-3130-9505}\\[2mm]
T. Kavaskar\\
{\small Department of Mathematics, Central University of Tamil Nadu,
Thiruvarur 610101, India}\\[2mm]
and \\
T. Tamizh Chelvam\\
{\small Manonmaniam Sundaranar University, Tirunelveli 627 012,
Tamilnadu, India}\\
{\small ORCID: 0000-0002-1878-7847}}
\date{24 July 2021}
\maketitle

\begin{abstract}
The zero-divisor graph of a finite commutative ring with unity is the graph
whose vertex set is the set of zero-divisors in the ring, with $a$ and $b$
adjacent if $ab=0$. We show that the class of zero-divisor graphs is universal,
in the sense that every finite graph is isomorphic to an induced subgraph of
a zero-divisor graph. This remains true for various restricted classes of
rings, including boolean rings, products of fields, and local rings. But
in more restricted classes, the zero-divisor graphs do not form a universal
family. For example, the zero-divisor graph of a local ring whose maximal
ideal is principal is a threshold graph; and every threshold graph
is embeddable in the zero-divisor graph of such a ring.
More generally, we give necessary and
sufficient conditions on a non-local ring for which its zero-divisor graph to be a threshold
graph. In addition, we show that there is a countable local ring whose zero-divisor
graph embeds the \emph{Rado graph}, and hence every finite or countable graph,
as induced subgraph. Finally, we consider embeddings in related graphs such as
the $2$-dimensional dot product graph.
\end{abstract}
\section{Introduction}
In this paper, ``ring'' means ``finite commutative ring with unity'',
while ``graph'' means ``finite simple undirected graph'' (except in the
penultimate section, where finiteness will be relaxed). The
\emph{zero-divisor graph} $\Gamma(R)$ of a ring $R$ has vertices the
zero-divisors in $R$ (the non-zero elements $a$ for which there exists $b\ne0$
such that $ab=0$), with an edge $\{a,b\}$ whenever $ab=0$. The zero-divisor graphs have been extensively studied in the past \cite{Akb,AN21, AND99,And1,And2,AND20,BMT16,BMT17,Kavas,Lag1, sel}. 

In this paper, we are interested in universal graphs. Let $G$ be a graph. A graph $H$ is said to be $G$-free if no subgraph of $H$ is isomorphic to $G$. A countable $G$-free graph $H$ is weakly universal if every
countable $G$-free graph is isomorphic to a subgraph of $H$, and strongly universal if every such graph is isomorphic to an induced subgraph of $H$. 
Similarly, a class of graphs $\mathcal C$ is said to be universal if every graph is an induced subgraph of a graph in the class $\mathcal C$. Universal graphs are well-studied and there is a vast literature spreading over past few decades \cite{hjp99, lcc}. 

Universal graphs for collection of graphs satisfying certain forbidden conditions also explored in the past \cite{av13,cs07,ks95}. In \cite{cg83}, the search for a graph $G$ on $n$ vertices and with minimum number of edges in which every tree on $n$ vertices is isomorphic to a spanning tree of $G$ is carried out.

Here we address the question: Which finite
graphs are induced subgraphs of the zero-divisor graph of some ring? The
answer turns out to be ``all of them'', but there are interesting questions
still open when we restrict the class of rings or the class of graphs.

\section{Zero divisor graphs are universal}
In this section, we show that the zero divisor graphs associated to various classes of rings are universal.

\subsection{The zero divisor graphs of Boolean rings are universal}
In this subsection we show that the intersection graphs are universal. Indeed we prove a stronger result: namely, every graph $G$ can be identified as an intersection graph for a natural choice of sets associated to $G$. We need this result in a couple of proofs in later sections.

\begin{defn} 
Let $X$ be a family of subsets of a set $A$. 
The \emph{intersection graph} on $X$, denoted by $\mathcal G(X)$, has
vertices the sets in the family, two vertices joined if they have non-empty
intersection.
\end{defn}
The following proposition is well known. 
\begin{proposition}\label{intersection}
Let $G$ be a finite graph. Then $G$ can be identified as an intersection graph for a natural choice of sets associated to $G$.
\end{proposition}
\begin{proof}
Let $G$ be a finite graph with vertex set $E$. Without loss of generality we assume that $G$ has no isolated vertices and isolated edges [See Remark \ref{rem1}]. The construction is simple. We represent a vertex $v$ of $G$ by the set $A_v$ of all edges of $G$ containing $v$. Then

$$
A_v \cap A_w = 	
\begin{cases}
\{e\} \text{ if $u$ is adjacent to $v$ by the edge $e$}\\
\emptyset \text{ otherwise}
\end{cases}
$$
Moreover, $A_v\ne A_w$ for all $v\ne w$, since (in the absence of isolated vertices and edges) two vertices cannot be adjacent with the same set of edges. Now, it is immediate that the graph $G$ and the intersection graph $\mathcal G(X)$ are isomorphic where $X = \{A_v : v \in V(G)\}$. \qed
\end{proof}

\begin{rem}\label{rem1}\rm
 If there are isolated vertices in $G$, they can be represented by non-empty
sets disjoint from the graph
obtained from the non-isolated vertices. If there are isolated edges in  $G$
we represent the vertices of such an edge by two distinct but intersecting sets
disjoint from the sets representing the rest of the graph. 
\end{rem}

Let $X$ be a finite set. The \emph{Boolean ring} on $X$ has as its elements
the subsets of $X$; addition is symmetric difference, and multiplication is
intersection. So the empty set is the zero element and $ab=0$ if and only if $a$ and $b$
are disjoint.

\begin{lem}\label{lemmaBoolean}
	Let $\Gamma(R)$ be the zero divisor graph of a Boolean ring $R$. A graph $G$ is an induced subgraph of $\Gamma(R)$ if, and only if, the complement graph of $G$ is isomorphic to an intersection graph $\mathcal G(X)$ for some collection of sets $X$.
\end{lem}
\begin{proof}
Suppose $G$ is an induced subgraph of the zero divisor graph of a Boolean ring $R$. Now $R$ is a subring of $(P(X),\sd,\cap)$
for some $X$. Therefore the vertices of $G$ are elements of $P(X)$ and two of them are adjacent if their intersection is empty. In otherwords, two of them are adjecent in the complement graph of $G$ if their intersection is non-empty. This shows that the complement graph of $G$ is an intersection graph. 
	
Conversely, consider an arbitrary intersection graph $G:=\mathcal G(X)$ for some collection of subsets $X$ of a set $A$. Then the vertices of $G$ are subsets of $A$ and two of them are adjacent if their intersection is non-empty. In otherwords, two of them are adjecent in the complement graph of $G$ if their intersection is empty. This shows that the complement graph of $G$ is an induced subgraph of the zero divisor graph of the Boolean ring $(P(A),\sd,\cap)$.
\qed
\end{proof}

\begin{thm}\label{theoremboolean}
The zero divisor graphs of Boolean rings are universal. That is, every finite
graph is an induced subgraph of the zero-divisor graph of a Boolean ring.
\end{thm}
\begin{proof}
	Given an arbitrary finite graph $G$, by Lemma \ref{lemmaBoolean}, it is enough to show that the complement graph of $G$ is an intersection graph. But this claim follows already from Proposition \ref{intersection}. This completes the proof. \qed
\end{proof}

\begin{cor} For every finite graph $G$, there exists positive integer $k$ such that $G$ is an induced subgraph of the zero-divisor graph of the ring $\mathbb{Z}_2\times\mathbb{Z}_2\times \ldots \times \mathbb{Z}_2$ ($k$-times).
\end{cor}
\begin{proof}
Note that every finite Boolean ring is isomorphic to $\mathbb{Z}_2\times\mathbb{Z}_2\times \ldots \times \mathbb{Z}_2$, for some positive integer $k$ and hence the result follows from Theorem \ref{theoremboolean}.
\end{proof}
\subsection{The zero divisor graphs of ring of integers modulo $n$ and the reduced rings are universal}
\begin{lem}
Let $X=\{1,2,\dots,m\}$. Then the zero divisor graph of the Boolean ring 
$(P(X),\sd,\cap)$
is an induced subgraph of the zero divisor graph of the ring $\mathbb{Z}_n$, the integers modulo $n$, where $n = p_1 p_2 \cdots p_m$ is a square-free integer
with $m$ prime divisors.
\end{lem}
\begin{proof}
Let $A$ and $B$ be two proper subsets of $X$. Then the following statements are
equivalent:
\begin{itemize}
\item $A$ and $B$ are adjacent in the the zero divisor graph of the Boolean
ring $(P(X),\sd,\cap)$;
\item $A$ and $B$ are disjoint;
\item $A^c \cup B^c = X$;
\item the product $n_{A^c} n_{B^c}$ is divisible by $n$, where
$n_A = \prod_{i \in A}p_i$.
\end{itemize}
Therefore the map which sends $A$ to $n_{A^c}$ defines an injective graph homomorphism of the zero divisor graph of the Boolean ring $(P(X),\sd,\cap)$
and the zero divisor graph of the ring $\mathbb Z_n$, the integers modulo $n$.
\qed
\end{proof}

\begin{thm}
Every finite graph is an induced subgraph of the zero-divisor graph of a
product of distinct finite fields.
\end{thm}

\begin{proof}
By the Chinese Remainder Theorem, if $n=p_1p_2\cdots p_m$ where $p_1$, $p_2$,
\dots, $p_m$ are distinct primes, the product of the fields $\mathbb{Z}_{p_i}$
is isomorphic to $\mathbb{Z}_n$, so the result follows from the preceding
theorem.\qed
\end{proof}

\begin{cor}
The class of zero-divisor graphs of the rings of integers modulo $n$ is universal. Further, we can even restrict the $n$ to be squarefree.
\end{cor}

\subsection{The zero divisor graphs of local rings are universal}

A \emph{local ring} is a ring $R$ with a unique maximal ideal $M$.
Any non-zero element of $M$ is a zero-divisor, and every element of
$R\setminus M$ is a unit. 
\begin{lem}
Let $R$ be a local ring with a unique maximal ideal $M$. Then there exists $r$ such that $M^r=\{0\}$.
\end{lem}
\begin{proof}
Since $R$ is a finite local ring, it is a local Artinian ring and hence $M$ is nilpotent.  \qed
\end{proof}

From the above lemma, we observe that, if $a\in M^s$ and $b\in M^t$ with $s+t\ge r$, then $ab=0$ and hence $a$ and
$b$ are adjacent in the zero-divisor graph.
This observation might suggest that the zero-divisor graph is a threshold graph (these graphs are defined in the next section), but we show that this is not so:
The problem is complicated by the fact that the converse of this observation is false in general.

In the following theorem, we show that in fact the zero-divisor graphs of local rings are universal.

\begin{thm}\label{t:loc_universal}
For every finite graph $G$, there is a finite commutative local ring $R$ with
unity such that $G$ is an induced subgraph of the zero-divisor
graph of $R$.
\end{thm}

\begin{proof}
Let the vertex set of $G$ be $ \{v_1,\ldots,v_n\}$. Take $R$ to be the quotient of
$F[x_1,\ldots,x_n]$ by the ideal $I$ generated by all homogeneous polynomials
of degree $3$ in $x_1,\ldots,x_n$ together with all products $x_ix_j$ for which
$\{v_i,v_j\}$ is an edge of $G$, where $F$ is a finite prime field of order $p$ and $x_1,\ldots,x_n$
are indeterminates. 

Let $M$ be the ideal of $R$ generated by the elements $x_1,\ldots,x_n$ (We
abuse notation by identifying polynomials in $x_1,\ldots,x_n$ with their
images in $R$). Thus $M$ is the set of elements represented by polynomials with
zero constant term. So $M$ is nilpotent and every element of $R\setminus M$
is a unit, so $M$ is the unique maximal ideal, and $R$ is a local ring.

Now the elements $x_1,\ldots,x_m$ of $M$ satisfy $x_ix_j=0$ if and only
if $\{v_i,v_j\}$ is an edge of $G$. To see this, note that the set
\[\{1,x_1,\ldots,x_n\}\cup\{x_ix_j:\{v_i,v_j\}\notin E(G)\}\]
is an $F$-basis for $R$. This completes the proof. \qed
\end{proof}

The above theorem shows that the zero divisor graph of a local ring need not be a threshold graph. We study threshold graphs extensively in the next section.

\section{When is the zero-divisor graph threshold?}

In preparation for the next section, we need some background on threshold
graphs. These were introduced by Chv\'atal and Hammer~\cite{ch} in 1977.

\begin{defn}
	A graph $G$ is a threshold graph if there exists $t \in \mathbb R$ and for each vertex $v$ a weight $w(v) \in \mathbb R$ such that $uv$ is an edge in $G$ if, and only if, $w(u)+w(v)>t$. 
\end{defn}

\begin{defn} 
A \textit{split graph} is a graph whose vertex set is the disjoint union of an
independent set and a clique, with arbitrary edges between them.
A \textit{nested split graph} (or NSG for short) is a split graph in which
we add cross edges in accordance to partitions of $U$ (the independent set) and
$V$ (the clique) into $h$ cells (namely, $U =U_1\cup U_2\cup \ldots \cup U_h$
and $V=V_1\cup V_2\cup \ldots \cup V_h$) in the following way: each vertex
$u\in U_i$ is adjacent to all vertices $v\in V_1\cup V_2\cup \ldots \cup V_i$.
The vertices $U_i \cup V_i$ form the $i$-th level of the NSG, and $h$ is the
number of levels. The NSG as described can be denoted by
$\mathrm{NSG}(m_1,m_2,\ldots, m_h; n_1,n_2,\ldots, n_h)$, where $m_i =|U_i|$
and $n_i =|V_i|\ (i=1,2,\ldots, h)$.
\end{defn}
The following theorem is well-known.
\begin{thm}[\cite{ch}]
For a finite graph $G$, the following three properties are equivalent:
\begin{enumerate}
\item $G$ is a threshold graph;
\item $G$ has no four-vertex induced subgraph isomorphic to $C_4$
(the cycle), $P_4$ (the path), or $2K_2$ (the matching);
\item $G$ can be, and built from the empty set by repeatedly adding
vertices joined either to nothing or to all other vertices;
\item $G$ is a nested split graph.
\end{enumerate}
\end{thm}

We thought originally that the zero-divisor graph of a local ring might be a  threshold graph. As we saw above, in Theorem \ref{t:loc_universal}, this is not true in general: see Example \ref{localex}. But there is
one case in which it holds, that where the maximal ideal is principal. If $p$
is a generator of $M$ as ideal, then every element of $M$ has the form $p^su$
where $u$ is a unit and $s>0$. If $u$ and $v$ are units, then $p^su.p^tv=0$
if and only if $s+t\ge r$, where $r$ is the nilpotent index of the ideal $M$.

\begin{defn}
A collection $\mathcal{C}$ of graphs is said to be \textit{threshold-universal} (for short, t-universal) if every threshold graph is an induced subgraph of a graph from $\mathcal{C}$.
\end{defn}

\begin{thm}
Let $R$ be a local ring whose maximal ideal $M$ is principal.
\begin{enumerate}
\item The zero-divisor graph of $R$ is a threshold graph.
\item Any threshold graph is an induced subgraph of some local ring whose
maximal ideal is principal. In other words the set of all zero divisor graphs of local rings with principal maximal form a t-universal collection of graphs.
\end{enumerate}
\end{thm}

\begin{prob}
We can similarly define other class of universal graphs such as c-universal, meaning ``cograph universal''. Can we find a class of rings whose zero-divisor
graphs are c-universal?
\end{prob}

\begin{proof}
(a) For the first part, let $R$ be a local ring with maximal ideal $M$. Let $r$
be the smallest integer such that $M^r=\{0\}$. Take threshold $t=r$, and set
$w(a)=i$ if $a\in M^i\setminus M^{i+1}$. By the remarks above,
$ab=0$ if and only if $w(a)+w(b)\ge r$, so the zero-divisor graph is a
threshold graph.

\medskip

(b) Let $G$ be an arbitrary threshold graph. Choose a prime which is sufficiently large (larger
than the number of vertices of $G$ is certainly enough). Let $m$ be the
number of stages required to dismantle $G$; embed the vertices of $G$
in $R=\mathbb{Z}/(p^{2m+1})$ as follows: if $a$ is removed as an isolated
vertex in round $i$, map it to an element of $p^iR\setminus p^{i+1}R$; if
it is removed as a vertex joined to all others in round $i$, map it to an
element of $p^{2m-i+1}R\setminus p^{2m-i+2}R$. (Each of these differences
contains at least $p-1$ elements, enough to embed all required vertices.)\qed
\end{proof}

We proved in Theorem \ref{t:loc_universal} that zero-divisor graphs of local rings are universal. In particular the graphs $P_4$ and $2K_2$ can occur as induced subgraphs of zero-divisor graphs of local rings.  The following example from~\cite{an} is a local ring whose zero-divisor graph is not threshold.

\begin{example}\label{localex}
Let $A=\mathbb{Z}_4[x,y,z]/M$, where $M$ is the ideal generated by
$\{x^2-2, y^2-2, z^2, 2x, 2y, 2z, xy, xz, yz-2\}$.  
In the zero-divisor graph of $A$, the induced subgraph on the
vertices $\{x, z, x+y, x+y+2\}$ is $2K_2$ and the induced subgraph of the
vertices $\{x, z+2, x+z, x+y\}$ is $P_4$.
\end{example}

The next result is a necessary condition for the zero divisor graph of a ring to
be threshold. Here $\Ann(x)=\{y\in R:xy=0\}$ is the \emph{annihilator} of $x$:
it is a non-zero ideal if $x$ is a zero-divisor.

\begin{thm}\label{t:threshold}
If $R$ is ring whose zero divisor graph is threshold, then for any two distinct
zero-divisors $x,y\in R$, the following holds:
\begin{enumerate}
\item $\Ann(x)\cap \Ann(y)\neq \{0\}$;
\item either $\Ann(x)\subseteq \Ann(y)$ or $\Ann(y)\subseteq \Ann(x)$;
\item if $\Ann(x)\subsetneq \Ann(y)$ and $xy\neq 0$, then $\langle \Ann(x)\rangle$ is a clique;
\item if $\Ann(x)\subsetneq \Ann(y)$ and $xy\neq 0$, then for any $a\in \Ann(x)$ and for any $b\in \Ann(y)\backslash \Ann(x)$, $ab=0$. 
\end{enumerate}
\end{thm}

\begin{proof}
(a) Suppose $\Ann(x)\cap \Ann(y)= \{0\}$, for some zero-divisors $x,y\in R$.
Then there exist $0\neq a\in \Ann(x)$ and $0\neq b\in \Ann(y)$ such that
$a\notin \Ann(y)$ and $b\notin \Ann(x)$. Hence, the subgraph induced by
$\{x,a,y,b\}$ is isomorphic to  $2K_2$, $P_4$ or $C_4$, a contradiction. 

\medskip

(b) Suppose that there exist non-zero elements $x,y\in R$ (necessarily 
zero-divisors) such that $\Ann(x)\not\subseteq\Ann(y)$ and
$\Ann(y)\not\subseteq\Ann(x)$. Then
again there exist $0\neq a\in \Ann(x)$ and $0\neq b\in \Ann(y)$ such that $a\notin \Ann(y)$ and $b\notin \Ann(x)$. Hence we get a similar contradiction in (a).

\medskip

(c) If $\Ann(x)\subsetneq \Ann(y)$ and $\langle \Ann(x)\rangle$ is not a clique,
then there exist $a,b\in \Ann(x)$ such that $ab\neq 0$ and hence the subgraph induced by 
$\{x,a,y,b\}$ is isomorphic to $C_4$, a contradiction. 

\medskip

(d) Similar to (c). \qed
\end{proof}

The following theorem is a characterization of non-local rings $R$ whose zero-divisor graph is threshold.

\begin{thm}
Let $R$ be a non-local ring. Then $\Gamma(R)$ is threshold if and only if $R\cong \mathbb{F}_2\times \mathbb{F}_q$, where $q\geq 3$ and $\mathbb{F}_q$ is the field with $q$ elements.
\end{thm}

\begin{proof}
If $R\cong \mathbb{F}_2\times \mathbb{F}_q$, then $\Gamma(R)\cong K_{1,\ q-1}$
(a star graph) and hence it is threshold. Suppose $\Gamma(R)$ is threshold and
$R=R_1\times R_2\times \ldots \times R_n$, where $n\geq 2$. If $n\geq 3$, then
$\Ann((1,1,0,0,\ldots,0))=\{0\}\times \{0\}\times R_3\times \ldots \times R_n$
and
$\Ann((0,0,1,0,\ldots, 0))=R_1\times R_2\times\{0\}\times R_4\times \ldots \times R_n$.
By Theorem~\ref{t:threshold}(b), we have $\Gamma(R)$ is not threshold. So $n=2$.
If one of the $R_i$ is not a field, say $R_1$, then there exist non-zero
elements $x,y\in R_1$ such that $xy=0$. Hence $\Ann((x,0))=\Ann(x)\times R_2$
and $\Ann((0,1))=R_1\times \{0\}$ and thus we have $\Gamma(R)$ is not threshold
by Theorem~\ref{t:threshold}(b). Therefore, both $R_1$ and $R_2$ are fields.
If $|R_i|>2$, for $i=1, 2$ then, by the same argument, $\Ann((1,0))$ is not a
subset of $\Ann((0,1))$ and $\Ann((0,1))$ is not a subset of $\Ann((1,0))$.\qed
\end{proof}
Next, we ask the following question about local rings. 
\begin{prob} For which local rings $(R,M)$ is the zero-divisor graph
threshold?
\end{prob}

For the rest of this section, we consider a local ring $(R,M)$, where $M$ is the
maximal ideal.

Since every finite commutative ring $R$ with unity is Noetherian, every ideal
of $R$ is finitely generated. In particular, the maximal ideal $M$ in a local
ring $(R,M)$ is finitely generated.    

\begin{proposition}
Let $M$ be the maximal ideal of $R$, and let $M$ be generated by $\{x_1,x_2,\ldots, x_k\}$, where $k\geq 2$. If $x_i^2=0$ for $1\leq i\leq k$ and $x_ix_j=0$ for $1\leq i<j \leq k$, then the zero-divisor graph of $R$ is threshold.  
\end{proposition}

\begin{proof}
Let $X=\sum_{i=1}^ka_ix_i\in V(\Gamma(R))$ and $Y=\sum_{i=1}^kb_ix_i\in V(\Gamma(R))$. Then $XY=0$ and hence $\Gamma(R)$ is complete. Therefore it is threshold.  \qed
\end{proof}

\begin{proposition} Let $M$ be the maximal ideal of $R$, and let $M$ be generated by $\{x_1,x_2,\ldots, x_k\}$, where $k\geq 2$, such that $x_ix_j=0$ for $1\leq i<j \leq k$. If $x_1^{n-1}\neq 0$, $x_1^n=0$ where $n\geq 3$ and $x_i^2=0$ for $2\leq i\leq k$, then $\Gamma(R)$ is threshold.
\end{proposition}

\begin{proof}
Using the definition of nested split graph, we prove the result. Set the level $h=\frac{n}{2}$ if $n$ is even and $h=\frac{n-1}{2}$ otherwise.

First we define, for $1\leq i\leq h-1$, 
\begin{eqnarray} \nonumber
U_i&=& \left\{
\begin{array}{l}
a_{i,1}x_1^i+a_{i,2}x_2+\ldots+a_{i,n}x_n\mid a_{i,1}\mbox{ is unit element of }R,\mbox{ and}\\
\qquad a_{i,j} \mbox{ is either zero or a unit element of } R, \mbox{ for } 2\leq j\leq n
\end{array}\right\}
\end{eqnarray}
and if $n$ is odd, define 
\begin{eqnarray} \nonumber
U_h&=& \left\{
\begin{array}{l}
a_{h,1}x_1^h+a_{h,2}x_2+\ldots+a_{h,n}x_n\mid a_{h,1}\mbox{ is unit element of }R,\mbox{ and}\\
\qquad a_{h,j} \mbox{ is either zero or a unit element of } R, \mbox{ for } 2\leq j\leq n
\end{array}\right\}
\end{eqnarray}
and if $n$ is even, define $U_h=\emptyset$.\\
Next we define, 
\begin{eqnarray} \nonumber
V_1&=& \left\{\begin{array}{l}
b_{1,1}x_1^{n-1}+b_{1,2}x_2+\ldots+b_{1,n}x_n\mid b_{1,j}\mbox{ is either zero or unit}\\
\qquad\mbox{element of } R,\mbox{ for } 1\leq j\leq n
\end{array}\right\}.
\end{eqnarray}
and for $2\leq i\leq h$, 
\begin{eqnarray} \nonumber
V_i&=& \left\{\begin{array}{l}
b_{i,1}x_1^{n-i}+b_{i,2}x_2+\ldots+b_{i,n}x_n\mid b_{i,1}\mbox{ is unit element of }R, \mbox{ and}\\
\qquad b_{i,j}\mbox{ is either zero or unit element of } R,\mbox{ for } 2\leq j\leq n
\end{array}\right\}.
\end{eqnarray}

Then clearly, $\mathop\bigcup\limits_{i=1}^h U_i$ is an independent set and $\mathop\bigcup\limits_{i=1}^h V_i$ is a complete subgraph of $\Gamma(R)$. Also they satisfy the definition of nested split graph and hence the graph is threshold. \qed
\end{proof}

\begin{proposition} If $M$ is the maximal ideal generated by $\{x_1,x_2,\ldots, x_k\}$, where $k\geq 2$ such that $x_1x_2=0$, $x_1^2\neq 0\neq x_2^2$, then $\Gamma(R)$ is not threshold. 
\end{proposition}

\begin{proof}
Clearly, $x_1\in \Ann(x_2)\backslash \Ann(x_1)$ and $x_2\in \Ann(x_1)\backslash \Ann(x_2)$ and hence $\Ann(x_1)$ is not a subset of $\Ann(x_2)$ and $\Ann(x_2)$ is not a subset of $\Ann(x_1)$. Thus $\Gamma(R)$ is not threshold, by Theorem \ref{t:threshold}. \qed
\end{proof}

\section{The zero-divisor graphs of local rings with countable cardinality are universal}

It is natural to wonder about embedding infinite graphs in zero-divisor graphs
of infinite rings. Here we consider the countable case.

Our tool will be the \emph{Rado graph}, or \emph{countable random graph}
$\mathcal{R}$: this was first explicitly constructed by Rado, but about the
same time Erd\H{o}s and R\'enyi proved that if a countable graph was selected
by choosing edges independently with probability $\frac{1}{2}$ from the
$2$-subsets of a countably infinite set, the resulting graph is isomorphic
to $\mathcal{R}$ with probability~$1$. Among many beautiful properties of
this graph (for which we refer to \cite{rg}), we require the following:
\begin{itemize}
\item $\mathcal{R}$ is the unique countable graph having the property that,
given any two finite disjoint sets $U$ and $V$ of vertices, there is a vertex
$z$ joined to every vertex in $U$ and to none in $V$.
\item Every finite or countable graph is embeddable in $\mathcal{R}$ as an
induced subgraph.
\end{itemize}
We refer to \cite{cl} for terminology and results on Model Theory, in
particular the Compactness and L\"owenheim--Skolem theorems.

\begin{thm}
There is a countable local ring with unity having the property that
every finite or countable graph is an induced subgraph of its zero-divisor
graph.
\end{thm}
\begin{proof}
It suffices to show that the Rado graph $\mathcal{R}$ can be embedded in the
zero-divisor graph of a countable ring, since every finite or countable
graph is an induced subgraph of the Rado graph.

We give two proofs of this. The first is simple and direct, using the
method we used in Theorem~\ref{t:loc_universal}. The second is non-constructive
but shows a technique which we hope will be of wider use.\\ 

\noindent \textbf{First proof.} Let $F$ be a field (for simplicity the field with
two elements). Let $R=F[X]/S$, where the set $X$ of indeterminates is bijective
with the vertex set of $\mathcal{R}$ (with $x_i$ corresponding to vertex $i$),
and $S$ is the ideal generated by all homogeneous polynomials of degree~$3$
in these variables together with all products $x_ix_j$ for which $\{i,j\}$ is
an edge of $\mathcal{R}$. Just as in the proof of Theorem~\ref{t:loc_universal},
the induced subgraph on the set $\{x_i+S:i\in V(\mathcal{R})\}$ induces a
subgraph of the zero-divisor graph of $R$.

The ring $R$ has the properties that it is a local ring (though its maximal
ideal is not finitely generated) and its automorphism group contains the
automorphism group of $\mathcal{R}$.\\

\noindent \textbf{Second proof.} We use basic results from model theory.
We take the first-order language of rings with unity together with an
additional unary relation $S$. Now consider the
following set $\Sigma$ of first-order sentences:
\begin{itemize}
\item[(a)] the axioms for a commutative ring with unity;
\item[(b)] the statement that every element of $S$ is a (non-zero)
zero-divisor;
\item[(c)] for each pair $(m,n)$ of non-negative integers, the sentence
stating that for any given $m+n$ elements $x_1,\ldots,x_m,y_1,\ldots,y_n$, all
satisfying $S$, there exists an element $z$ such that $z$ satisfies $S$,
that $x_iz=0$ for $i=1,\ldots,m$, and that $y_jz\ne0$ for $j=1,\ldots,n$.
\end{itemize}

We claim that any finite set of these sentences has a model. Any
finite subset of the sentences in (c) are satisfied in some finite graph
(taking ``product zero'' to mean ``adjacent''), with $S$ satisfied by the
vertices used in the embedding: indeed, a sufficiently large finite random
graph will have this property. By our earlier results, this finite graph is
embeddable as an induced subgraph in the zero-divisor graph of some finite
commutative ring with unity.

By the First-Order Compactness Theorem, the entire set $\Sigma$ has a model $R$.
This says that the set of ring elements satisfying $S$ induces a subgraph
isomorphic to Rado's graph $\mathcal{R}$ in the zero-divisor graph of $R$. (The
sentences under (c) are first-order axioms for $\mathcal{R}$.) But
$\mathcal{R}$ contains every finite or countable graph as an induced subgraph.

Now the downward L\"owenheim--Skolem theorem guarantees that there is a
countable ring whose zero-divisor graph contains $\mathcal{R}$, and hence all
finite and countable graphs, as induced subgraphs. 
\end{proof}
\begin{prob} Is there a local ring whose maximal ideal is finitely
generated as ideal and whose zero-divisor graph embeds the Rado graph as an
induced subgraph?
\end{prob}

\begin{prob} Find the smallest number $N$ (in terms of $n$ and $m$)
such that, if $G$ is a graph with $n$ vertices and $m$ edges, then $G$ is
an induced subgraph of the zero divisor graph of a ring (commutative with
unity) of order at most $N$.
\end{prob}
Let $f(n,m)$ be this number. Our construction using Boolean rings shows that $N\le 2^p$, where $p$ is the
least number of points in a representation of $G$ as an intersection graph.
Moreover, we constructed an intersection representation with $p=m+a+b$, 
where $a$ and $b$ are the numbers of isolated vertices and edges in $G$. 
This gives an upper bound for $f(n,m)$. 
\begin{prob}
Is this best possible?
\end{prob}
We can ask similar questions for subclasses of rings (such as local rings),
or for variants of the zero-divisor graph.

\section{Other graphs from rings}

There are several natural classes of graphs containing the threshold graphs.
These include split graphs (defined earlier), chordal graphs (containing no
induced cycle of length greater than~$3$), cographs (containing no induced
path on four vertices) and perfect graphs (containing no induced odd cycle
of length at least~$5$ or complement of one). For each of these classes
$\mathcal{C}$, we can ask:

\begin{prob} For which commutative rings with unity does the 
zero-divisor graph belong to $\mathcal{C}$?
\end{prob}

Another very general problem is to examine the induced subgraphs of various
generalizations of zero-divisor graphs, such as the extended zero-divisor
graphs and the trace graphs~\cite{Trace,Siva1,Siva2,Siva3,Siva4}.

As a contribution to this problem, here is an example of a kind of universality
question that can be asked when we have graphs defined from algebraic
structures where one is a subgraph of the other. The pattern for this theorem 
is \cite[Theorem 5.9]{c22}, the analogous result for the enhanced power graph
and commuting graph of a group.

Let $A$ be a commutative ring with unity. We define the zero-divisor graph
of $A$ to have as vertices all the non-zero elements of $A$, two vertices
$a$ and $b$ joined if $ab=0$. (This is not the usual definition since we
don't restrict just to zero-divisors: vertices which are not zero-divisors are
isolated. This doesn't affect our conclusion.) Given a positive integer $m$, we
define the $m$-dimensional dot product graph to have vertices all the
non-zero elements of $A^m$ with two vertices $(a_1,a_2,\ldots,a_m)$ and
$(b_1,b_2,\ldots,b_m)$ joined if $a.b=0$, where
\[a \cdot b=a_1b_1+a_2b_2+\cdots+a_mb_m.\]

Note that $A^m$ is a ring, with the product $*$ given by
\[a*b=(a_1b_2,a_2b_2,\ldots,a_mb_m).\]
It is clear that $a*b=0$ implies $a.b=0$, so the zero-divisor graph of $A^m$
is a spanning subgraph of the $m$-dimensional dot product graph of $A$.

We prove the following:

\begin{thm}
Take the complete graph on a finite set $X$, with the edges coloured red, green
and blue in any manner whatever. Then there is a ring $A$ and an embedding of
$X$ into $A^2$ such that
\begin{itemize}\itemsep0pt
\item the red edges are edges of the zero-divisor graph of $A^2$;
\item the green edges are edges of the $2$-dimensional dot product graph of $A$
but not of the zero-divisor graph of $A^2$;
\item the blue edges are not edges of the $2$-dimensional dot product graph
of $A$.
\end{itemize}
\label{t:three}
\end{thm}

\begin{proof}
First, by enlarging $X$ by at most four points joined to all others with
blue or green edges, we can assume that neither the blue nor the green 
subgraphs have isolated vertices or edges.

Let $P$ be the set of blue or green edges. For each vertex $v\in X$, define
a pair $(S(v),T(v))$ of subsets of $P$ by the rule that $S(v)$ is the set
of green edges containing $v$, while $T(v)$ is the set of blue or green edges
containing $v$. The assumption in the previous paragraph shows that the map
$\theta:v\mapsto(S(v),T(v))$ is one-to-one.

Now let $A$ denote the Boolean ring on $P$. Then $\theta$ is an embedding of
$X$ into $A\times A\setminus\{0\}$. We claim that this has the required
property.
\begin{itemize}
\item Suppose that $e=\{v,w\}$ is a red edge. Then $S(v)\cap S(w)=\emptyset$
and $T(v)\cap T(w)=\emptyset$; so $\theta(v)*\theta(w)=0$, whence 
$\theta(v)$ and $\theta(w)$ are joined in the zero divisor graph of $A^2$.
\item Suppose that $e=\{v,w\}$ is green. Then $S(v)\cap S(w)=\{e\}$ and
$T(v)\cap T(w)=\{e\}$; so $\theta(v)*\theta(w)\ne0$ but
$\theta(v).\theta(w)=0$. Thus $\theta(v)$ and $\theta(w)$ are joined in the
dot product graph but not the zero divisor graph.
\item Suppose that $e=\{v,w\}$ is blue. Then $S(v)\cap S(w)=\emptyset$
and $T(v)\cap T(w)=\{e\}$, so $\theta(v).\theta(w)=\{e\}\ne0$. So $\theta(v)$
and $\theta(w)$ are not joined in the dot product graph.
\end{itemize}
The theorem is proved.\qed
\end{proof}

\begin{rem}
This theorem has several consequences:
\begin{itemize}
\item By ignoring the distinction between green and blue, we have another
construction showing the universality of the zero-divisor graphs of rings.
\item By ignoring the distinction between red and green, we have shown the
universality of the $2$-dimensional dot product graphs of rings.
\item By ignoring the distinction between red and blue, we have shown that
the graphs obtained from the $2$-dimensional dot product graph of $A$ by
deleting the edges of the zero-divisor graph of $A^2$ are universal.
\end{itemize}
\end{rem}
The theorem suggests several questions:
\begin{prob}
\begin{itemize}\itemsep0pt
\item Can we restrict the ring $A$ to a special class such as local rings?
\item Can we prove similar results for other pairs of graphs?
\item Can we prove similar results for more than two graphs?
\end{itemize}
\end{prob}

\noindent \textbf{Acknowledgment.} 
\small{This work began in the Research Discussion on Groups and Graphs, organised by
Ambat Vijayakumar and Aparna Lakshmanan at CUSAT, Kochi, to whom we express
our gratitude. We are also grateful to the International Workshop on Graphs
from Algebraic Structures, organised by Manonmaniam Sundaranar University and
the Academy of Discrete Mathematics and Applications, for the incentive to
complete the work and the inspiration for Theorem~\ref{t:three}.} For T. Tamizh Chelvam, this research was supported by the University Grant Commissions Start-Up Grant, Government of India grant No. F. 30-464/2019 (BSR) dated 27.03.2019, while for the last author, it was supported by CSIR Emeritus Scientist Scheme (No. 21 (1123)/20/EMR-II) of  Council of Scientific and Industrial Research, Government of India. Peter J. Cameron acknowledges the Isaac Newton Institute for Mathematical Sciences, Cambridge, for support and hospitality during the programme \textit{Groups, representations and applications: new perspectives} (supported by \mbox{EPSRC} grant no.\ EP/R014604/1), where he held a Simons Fellowship.

\end{document}